\documentclass[a4paper]{amsart}

\usepackage{mathptmx}
\usepackage{graphicx}
\usepackage{amsmath}
\usepackage{amsthm}
\usepackage{amssymb}
\usepackage{enumerate}
\usepackage{mathrsfs} 
\usepackage{hyperref}
\usepackage[nooneline, FIGTOPCAP]{subfigure}
\usepackage{tikz-cd}

\theoremstyle{plain}
  \newtheorem{theorem}{Theorem}[section]
  \newtheorem{lemma}[theorem]{Lemma}
  \newtheorem{proposition}[theorem]{Proposition}
  
  \newtheorem{claim}{Claim}
  
\theoremstyle{definition}
  \newtheorem{definition}[theorem]{Definition}
    \newtheorem{example}[theorem]{Example}

\newcommand{\D}{\mathcal{D}}
\newcommand{\Z}{\mathbb{Z}}

\newtheorem*{proofClaim}{Proof of the claim}{\itshape}{\rmfamily}

\begin{document}

\title{Banchoff's sphere and branched covers over the trefoil}

\author{\'Alvaro Lozano-Rojo
  \and
  Rub\'en Vigara}

\address{Centro Universitario de la Defensa Zaragoza, Academia General Militar
Carretera de Huesca s/n. 50090 Zaragoza, Spain --- IUMA, Universidad de
Zaragoza}
      \email{alozano@unizar.es}
      \email{rvigara@unizar.es}


\begin{abstract}
  A filling Dehn surface in a $3$-manifold $M$ is a generically immersed surface
  in $M$ that induces a cellular decomposition of $M$. 
  Given a tame link $L$ in $M$ there is a
  filling Dehn sphere of $M$ that ``trivializes'' (\emph{diametrically 
  splits}) it. This allows to construct filling Dehn surfaces in the coverings 
  of $M$ branched over $L$.
  It is shown that one of the simplest filling Dehn spheres
  of $S^3$ (Banchoff's sphere) diametrically splits the trefoil knot.
  Filling Dehn spheres, and their Johansson diagrams, are constructed
  for the coverings of $S^3$ branched over the trefoil.
  The construction is explained in detail. Johansson diagrams for 
  generic cyclic coverings and for the simplest
  locally cyclic and irregular ones are constructed explicitly, providing
  new proofs of known results about cyclic coverings and the $3$-fold irregular
  covering over the trefoil.\\
  \\
  \emph{Dedicated to Prof. Maite Lozano on her 70th anniversary.}
  
\end{abstract}

\thanks{Partially supported by the European Social Fund and Diputaci\'on General de 
  Arag\'on (Grant E15 Geometr{\'\i}a) and by MINECO grants 
  MTM2013-46337-C2, MTM2016-77642-C2, MTM2013-45710-C2 and MTM2016-76868-C2-2-P.
}

\maketitle

\section{Introduction}

Filling Dehn surfaces and their Johansson diagrams 
were introduced in~\cite{Montesinos}, following ideas of~\cite{Haken1}, as a 
new way to represent closed orientable $3$-mani\-folds. After~\cite{Montesinos} 
some works have appeared on the 
subject~\cite{Amendola09,racsam,spectrum,peazolibro,Anewproof,RHomotopies,tesis}.

In~\cite{knots}, the authors propose a general framework in which filling Dehn 
surfaces can be applied to knot theory. Any knot (or link) in any 
$3$-manifold can be nicely intersected (\emph{split}) by a filling Dehn sphere. 
This filling Dehn sphere appears to be an interesting tool for studying the 
branched coverings over the knot, because the splitting
sphere has ``nice lifts'' to these branched coverings,
in a similar way as the Heegaard surface of a $(g,1)$-decomposition of the 
knot~\cite{Cattabriga-Mulazzani,cristoforietal,doll}. In~\cite{knots}, this is exemplified 
with the simplest of all knots: the unknot. The techniques 
of~\cite{knots} are applied here to the next knot in increasing 
complexity after the unknot: the trefoil knot.

In Section~\ref{sec:Dehn-surfaces-Johansson-diagrams} we introduce the basic 
definitions and notation about filling Dehn surfaces. 
Section~\ref{sec:splitting} recalls the tools introduced in~\cite{knots}. We 
refer to~\cite{knots,peazolibro} and references therein for more details on 
this subject. 

In Section~\ref{sec:trefoil}, it is 
shown that one of the simplest filling Dehn spheres of $S^3$
(\emph{Banchoff's sphere}) splits the trefoil knot. This is used in the 
subsequent sections to study covers of $S^3$ branched over the trefoil.
In Section~\ref{sec:branched-covers} we study the cyclic branched covers,
obtaining a new proof of Theorem~2.1 of~\cite{CHK} that asserts that the 
$3$-manifolds introduced in~\cite{sieradski} coincide with the
cyclic branched covers of the trefoil knot. Section~\ref{sec:other-examples}
gives some information about other type of coverings, as the locally cyclic
(Section~\ref{sub:trefoil-locally-cyclic}) and irregular ones 
(Section~\ref{sub:trefoil-non-cyclic}). In particular,
in Section~\ref{sub:trefoil-non-cyclic} we give another proof of the well known
result that asserts that the irregular $3$-fold covering of $S^3$ branched over the trefoil 
is $S^3$~\cite{Burde,GH,MontesinosTesis}.

\section{Dehn surfaces and their Johansson's diagrams}
\label{sec:Dehn-surfaces-Johansson-diagrams} 

Throughout the paper all $3$-manifolds are assumed to be closed and 
orientable, that is, compact connected and without boundary. On the
contrary, surfaces are assumed to be compact, orientable and without boundary, 
but they could be disconnected. 
All objects are assumed to be in the smooth 
category: manifolds have a differentiable structure and all maps are smooth.

Let $M$ be a 3-manifold.

A subset $\Sigma\subset M$ is a \emph{Dehn surface} in $M$~\cite{Papa} if there 
exists a surface $S$ and a general
position immersion $f:S\rightarrow M$ such 
that $\Sigma=f\left(S\right)$. If this is the case, the surface $S$ is the 
\emph{domain} of $\Sigma$ and it is said that $f$ \emph{parametrizes} $\Sigma$.
If $S$ is a $2$-sphere, then $\Sigma$ is a \emph{Dehn sphere}. 

Let $\Sigma$ be a Dehn surface in $M$ and consider a parametrization $f:S\to M$ 
of $\Sigma$. The \emph{singularities} of $\Sigma$ are the points $x\in\Sigma$ 
such that $\#f^{-1}(x)>1$, and they are divided into \emph{double points} where 
two sheets of $\Sigma$ intersect transversely ($\#f^{-1}(x)=2$), and 
\emph{triple points} where three sheets of $\Sigma$ intersect transversely 
($\#f^{-1}(x)=3$). The singularities of $\Sigma$ form the \emph{singularity 
set} $S(\Sigma)$ of $\Sigma$. We denote by $T(\Sigma)$ the set of triple points 
of $\Sigma$.
The connected components of $S(\Sigma)-T(\Sigma)$, $\Sigma-S(\Sigma)$ and
$M-\Sigma$ are the \emph{edges}, \emph{faces} and \emph{regions} of $\Sigma$,
respectively.

In the following a \emph{curve} in $S$, $\Sigma$ or $M$ is the image of an
immersion from $\mathbb{S}^1$ or $\mathbb{R}$ 
into $S$, $\Sigma$ or $M$, respectively. A \emph{double curve} of $\Sigma$ is a 
curve in $M$ contained in $S(\Sigma)$.

The preimage under $f$ of the 
singularity set of $\Sigma$, together with the information about how its 
points become identified by $f$ in $\Sigma$ is the \emph{Johansson diagram} 
$\D$ of $\Sigma$, see~\cite{Johansson1,Montesinos}. Two points of $S$ 
are \emph{related} by $f$ if they project onto the same point of $\Sigma$.

Because $S$ is compact and without boundary, double curves are closed and there 
is a finite number of them, and the number of triple points is also finite. Since $S$ and $M$ are 
orientable, the preimage under $f$ of a double curve of $\Sigma$ is the union 
of two different closed curves in $S$. These two curves are \emph{sister 
curves} of $\D$. Thus, the Johansson diagram of $\Sigma$ is composed by an even 
number of different closed curves in $S$. We identify $\D$ with the 
set of different curves that compose it. For any curve $\alpha\in\D$ we denote 
by $\tau \alpha$ the sister curve of $\alpha$ in $\D$. This defines a free 
involution $\tau:\D\rightarrow\D$, the \emph{sistering} of $\D$, that sends 
each curve of $\D$ into its sister curve in $\D$.

The curves of $\D$ transversely meet others or themselves at the 
\emph{crossings} of $\D$. The crossings of $\D$ are the preimage under $f$ of 
the triple points of $\Sigma$. If $P$ is a triple point of $\Sigma$, the three 
crossings of $\D$ in $f^{-1}(P)$ form \emph{the triplet of} $P$.

The Dehn surface $\Sigma\subset M$ \emph{fills} $M$ if it defines a 
cell-decomposition of $M$ whose $0$-, $1$- and $2$-dimensional skeletons are 
$T(\Sigma)$, $S(\Sigma)$, and $\Sigma$ respectively~\cite{Montesinos}. If 
$\Sigma$ fills $M$ and the domain $S$ of $\Sigma$ is connected, then it is 
possible to build $M$ out of the Johansson diagram $\D$ of $\Sigma$. Since 
every $3$-manifold has a filling Dehn sphere~\cite{Montesinos}, Johansson 
diagrams of filling Dehn spheres represent all closed orientable $3$-manifolds.

A special case of filling Dehn surfaces is when the domain $S$ of the filling 
Dehn surface $\Sigma$ is a disjoint union of $2$-spheres. In this case, we say 
that $\Sigma$ is a \emph{filling collection of spheres}.

\section{Splitting knots with filling Dehn spheres}
\label{sec:splitting}

In the following paragraphs we summarize some definitions and results 
from~\cite{knots}. Results are stated without proof.

Let $L$ be a tame knot or link in a $3$-manifold $M$, and let $\Sigma$ be a 
filling Dehn surface of $M$.

\begin{definition}
  The Dehn surface $\Sigma$ \emph{splits} $L$ if:
  \begin{enumerate}
    
    \item $L$ intersects $\Sigma$ transversely in a finite set of non-singular 
      points of $\Sigma$;
    
    \item for each region $R$ of $\Sigma$, if the intersection $R\cap L$ is 
      non-empty it is exactly one arc, unknotted in $R$; and
    
    \item for each face $F$ of $\Sigma$, the intersection $F\cap L$ contains at 
      most one point.
    
  \end{enumerate}

The Dehn surface $\Sigma$ \emph{diametrically splits} $L$ if it splits $L$ and 
it intersects each connected component of $L$ exactly twice. 
\end{definition}

\begin{theorem}
  There is a filling Dehn sphere of $M$ that diametrically splits $L$. \qed
\end{theorem}

Assume that $\Sigma$ splits $L$.

Our interest in filling Dehn surfaces that (diametrically) split knots relies 
on the following result. Let $p:\widehat{M}\to M$ be a finite sheeted branched 
covering with downstairs branching set $L$, and take 
$\widehat{\Sigma}=p^{-1}(\Sigma)$.

\begin{theorem}
  \label{thm:filling-lift-to-filling}
  The Dehn surface $\widehat{\Sigma}$ fills $\widehat{M}$. Moreover, if $L$ 
  is a knot and $\Sigma$ diametrically splits $L$ then $\widehat{\Sigma}$ is a 
  filling collection of spheres in $\widehat{M}$, and it is a Dehn sphere if and 
  only if $p$ is locally cyclic. \qed
\end{theorem}

Recall that a $n$-fold covering $p$ branched over a knot $L$ is \emph{locally 
cyclic} if its monodromy map $\rho$ sends knot meridians onto 
$n$-cycles~\cite[p.~209]{ST}. This is equivalent to say that $p:p^{-1}(L)\to L$ 
is a homeomorphism.

For the study of branched coverings over $L$ it is essential to know 
its group, i.e. the fundamental group of $M-L$. This can be done also using 
$\Sigma$. If $R_1,\ldots,R_m$ are all the different regions of $\Sigma$ 
disjoint from $L$ and we take a point $Q_i$ in each of these regions, 
$\Sigma-L$ is a strong deformation retract of 
$M-\left(L\cup \{Q_1,\ldots,Q_m\}\right)$. Hence

\begin{proposition}
  The fundamental groups of $M-L$ and $\Sigma- L$ are isomorphic.\hfill$\qed$
\end{proposition}
  
If $f:S\to M$ is a parametrization of $\Sigma$, the pair $(\D,f^{-1}(L))$, 
where $\D$ is the Johansson diagram of $\Sigma$, is a
\emph{Johansson diagram of} $L$.

\begin{proposition}\label{prop:diagram-represent-knots}
  The pair $(M,L)$ can be recovered from a Johansson diagram of $L$. In 
  particular, if $L'$ is a link in a $3$-manifold $M'$ such that $L$ and $L'$
  have identical Johansson diagrams, there is a homeomorphism between $M$ and 
  $M'$ that maps $L$ onto $L'$. \qed
\end{proposition}
Thus, all the information about $L$ is codified 
  in its Johansson diagram.

A presentation of the fundamental group of a Dehn surface in terms of its 
Johansson diagram was introduced in~\cite{spectrum} (cf.~\cite{peazolibro}). 
Although the presentation given there is stated for Dehn surfaces of genus $g$, 
it is valid also in a more general context, including the case  of $\Sigma-L$ 
where the domain surface is a punctured sphere. The generators of this 
presentation are of two kinds: \emph{surfacewise} generators and 
$\D$-\emph{dual} generators.

Set $M_L := M-L$ and $S_L := S-f^{-1}(L)$. Take a non-singular point $x$ of 
$\Sigma$ as the base point of the fundamental group $\pi_L:=\pi_1(M_L,x)$ of 
$M_L$. We also denote by $x$ the preimage of $x$ under $f$, and we choose it as 
the base point of the fundamental group $\pi_{S_L}:=\pi_1(S_L,x)$ of $S_L$.

The surfacewise generators of $\pi_L$ are obtained by pushing forward a 
generating set of $\pi_{S_L}$ to $M$ through $f$: if
$\gamma_1,\ldots,\gamma_k$ are representatives of a set of generators 
of $\pi_{S_L}$, then $f\circ\gamma_1,\ldots,f\circ\gamma_k$ are 
representatives of a set of surfacewise generators of $\pi_L$.

Let $\alpha$ and $\beta$ be sister curves of the diagram $\D$. Consider two 
paths $a$ and $b$ in $S_L$ starting at $x$ and ending at points
on $\alpha$ and $\beta$ respectively. Assume that the endpoints of $a$ 
and $b$ are related by $f$ 
and that they are not crossings of $\D$. In this case, 
$\alpha^{\#}=(f\circ a)\,(f\circ b)^{-1}$ is a loop in $\Sigma-L$ which is said 
to be \emph{dual to} $\alpha$. The inverse loop 
$\beta^{\#}=(\alpha^{\#})^{-1}=(f\circ b)\,(f\circ a)^{-1}$ is dual to $\beta$. 
After repeating this construction for each pair of sister curves of $\D$ we 
obtain the set $\D^{\#}$ of $\D$-dual generators of $\pi_L$.

\begin{proposition}[\cite{spectrum}]
  Surfacewise generators and  $\D$-dual generators generate $\pi_L$.\hfill$\qed$
\end{proposition}

Thus, surfacewise and $\D$-dual generators lead to a presentation of $\pi_L$. 
The relators associated to this presentation are detailed in~\cite{spectrum}.

If $p:\widehat{M}\to M$ is an $n$-fold ($n<\infty$) covering of $M$ branched over 
$L$, according to Theorem~\ref{thm:filling-lift-to-filling}, $\Sigma$ lifts to 
a filling Dehn surface $\widehat{\Sigma}$ of $\widehat{M}$. We want to construct 
the Johansson diagram $\widehat{\D}$ of $\widehat{\Sigma}$. This construction is 
specified in Algorithm 3.4 of~\cite{knots}. The presentation of $\pi_L$ in 
terms of surfacewise and $\D$-dual generators fits quite well to this purpose. 
In short form:
\begin{itemize}
  
  \item surfacewise generators of $\pi_L$ allow to construct the domain surface 
    $\widehat{S}$ of $\widehat{\Sigma}$; and
  
  \item $\D$-dual generators allow to decide the sistering between the curves 
    of $\widehat{\D}$.
  
\end{itemize}

By~\cite{spectrum}, there is a commutative diagram
\[
  \begin{tikzcd}
    \widehat{S} \rar{\widehat{f}} \dar[swap]{p_S} & \widehat{M} \dar{p} \\
    S \rar[swap]{f} & M
  \end{tikzcd}
\]
where $\widehat{f}$ is a parametrization of $\widehat{\Sigma}$, and $p_S$ is an $n$-fold
branched covering with branching set $f^{-1}(L)$.

Take $p^{-1}(x)=\{x_1,\ldots,x_n\}$, and also denote by $\{x_1,\ldots,x_n\}$ 
the corresponding points in $\widehat{S}$. If $\rho:\pi_L\to \Omega_n$ is the 
monodromy homomorphism associated with $p$, where $\Omega_n$ is the group of 
permutations of the set $\{x_1,\ldots,x_n\}$, the monodromy homomorphism 
$\rho_S:\pi_{S_L}\to \Omega_n$ associated with $p_S$ verifies 
$\rho_S=\rho \circ f_*$, where $f_*: \pi_{S_L}\to \pi_L$ is the homomorphism 
induced by $f$.

Since $f_*$ sends a set of generators of $\pi_{S_L}$ onto the surfacewise 
generators of $\pi_L$, $\rho_S$ is essentially the same as $\rho$ restricted to 
the surfacewise generators of $\pi_L$. Hence, the knowledge of $\rho$ allows to
construct $\widehat{S}$.

Once $\widehat{S}$ is constructed, the curves of $\widehat{\D}$ are the lifts to 
$\widehat{S}$ of the curves of 
$\D$. Consider the pair of sister curves $\alpha$ and $\beta$ of $\D$ and their 
associated paths $a$ and $b$ as before. Let $a_i$ and 
$b_i$ be the lifts of $a$ and $b$ respectively to $\widehat{S}$ based at $x_i$, and 
let $\alpha_i$ and $\beta_i$ be the lifts of $\alpha$ and $\beta$ passing
through the endpoint of $a_i$ and $b_i$ respectively, with $i=1,\ldots,n$.
The monodromy map $\rho$ assigns to $\alpha^{\#}$ the permutation 
$\rho(\alpha^{\#})$ of $\{x_1,\ldots,x_n\}$ given by
\[
  \rho(\alpha^{\#})(x_i)=x_{j}\iff 
    \text{\emph{the lift of $\alpha^{\#}$ starting at $x_i$ ends at $x_{j}$}.}
\]
By the construction of $\alpha^{\#}$, the right-hand side of the previous 
equivalence is indeed equivalent to say that the endpoints of $a_i$ and 
$b_j$ are related by $\widehat{f}$. 
Therefore, $\alpha_i$ and $\beta_j$ are sister curves in $\widehat{\D}$ and they 
must be identified in such a way that the endpoints of $a_i$ and $b_j$ are related 
by $\widehat{f}$. Hence, $\rho(\alpha^{\#})$ tells us how the
lifts of $\alpha$ to $\widehat{S}$ must be identified with the lifts of 
$\beta$ to $\widehat{S}$. Repeating the same argument for the rest of $\D$-dual 
generators, all the identifications between the curves of $\widehat{\D}$ are 
established.

\section{Banchoff's sphere and the trefoil knot}
\label{sec:trefoil}

\begin{figure}
  \centering
  \subfigure[]{
    \includegraphics[width=0.25\textwidth]{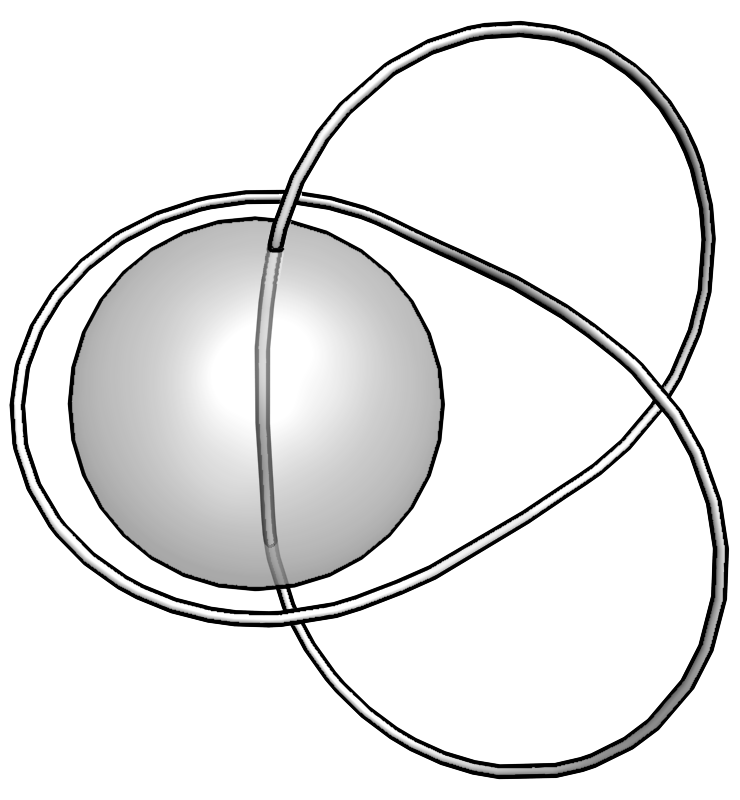}
  }
  \hfill
  \subfigure[]{
    \includegraphics[width=0.35\textwidth]{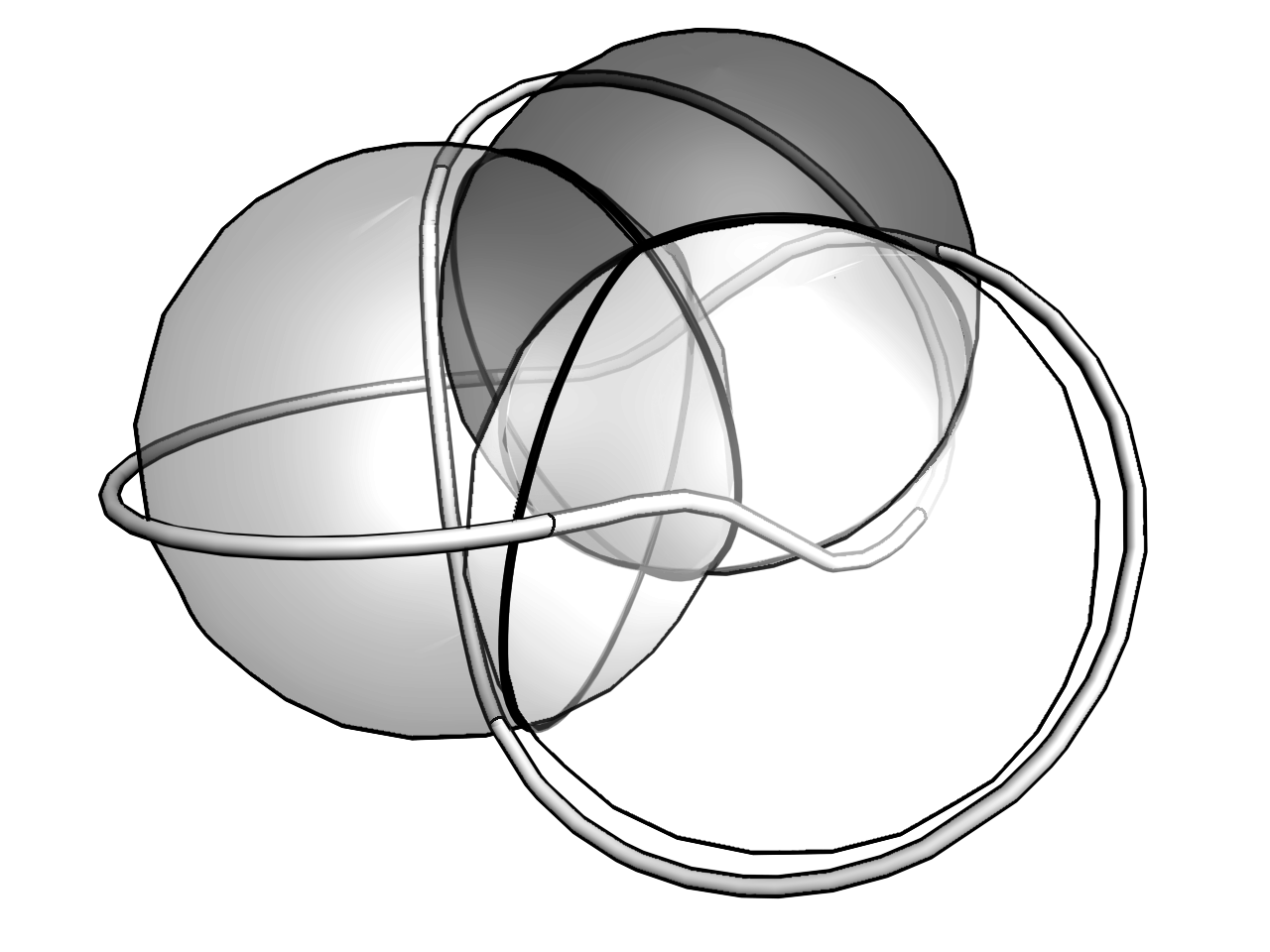}
  }
  \hfill
  \subfigure[]{
    \includegraphics[width=0.35\textwidth]{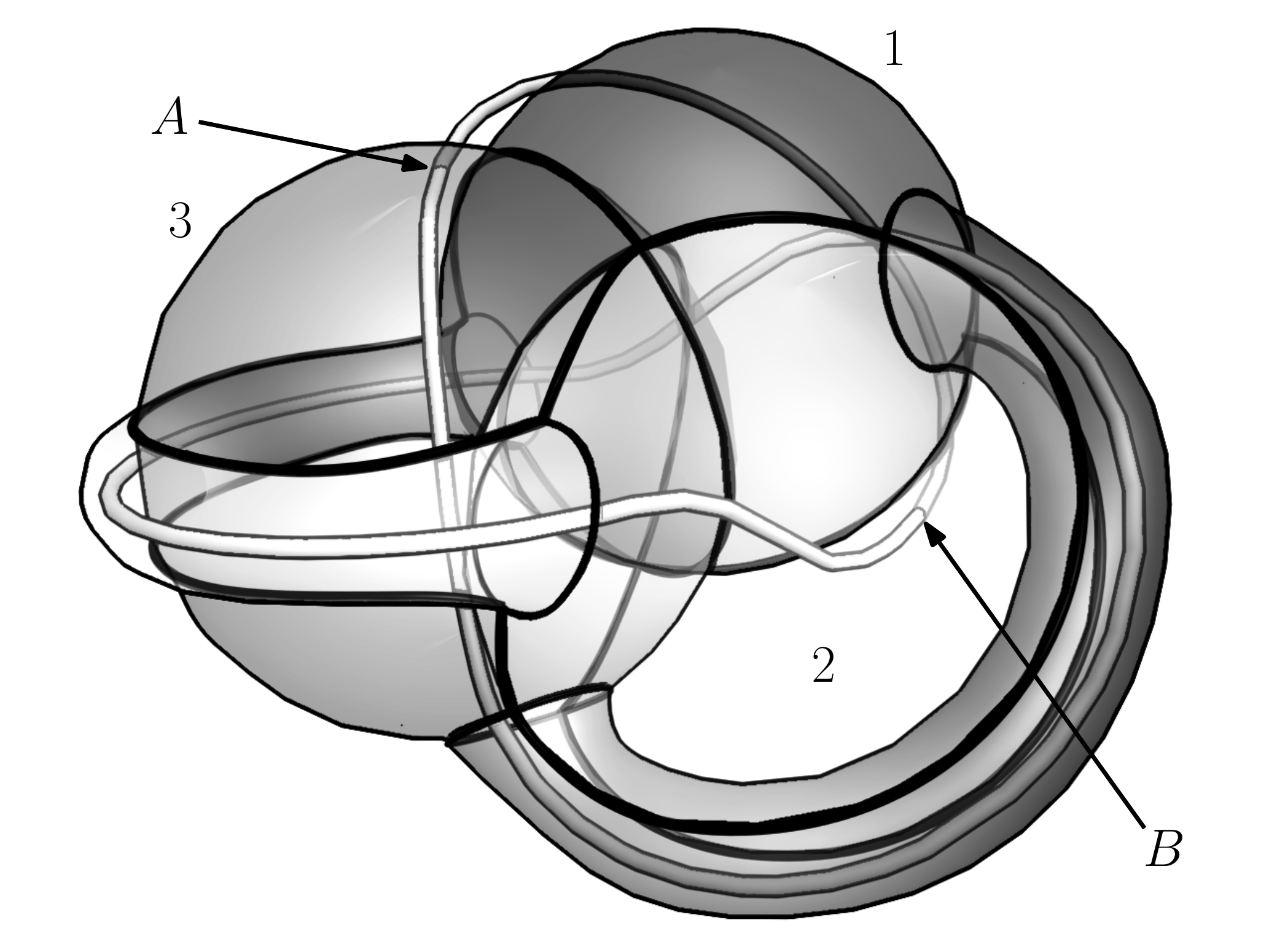}
  }

  \caption{Banchoff's sphere splitting the trefoil}
  \label{fig:trefoil-banchoff}
\end{figure}

\begin{figure}
  \centering
  \includegraphics{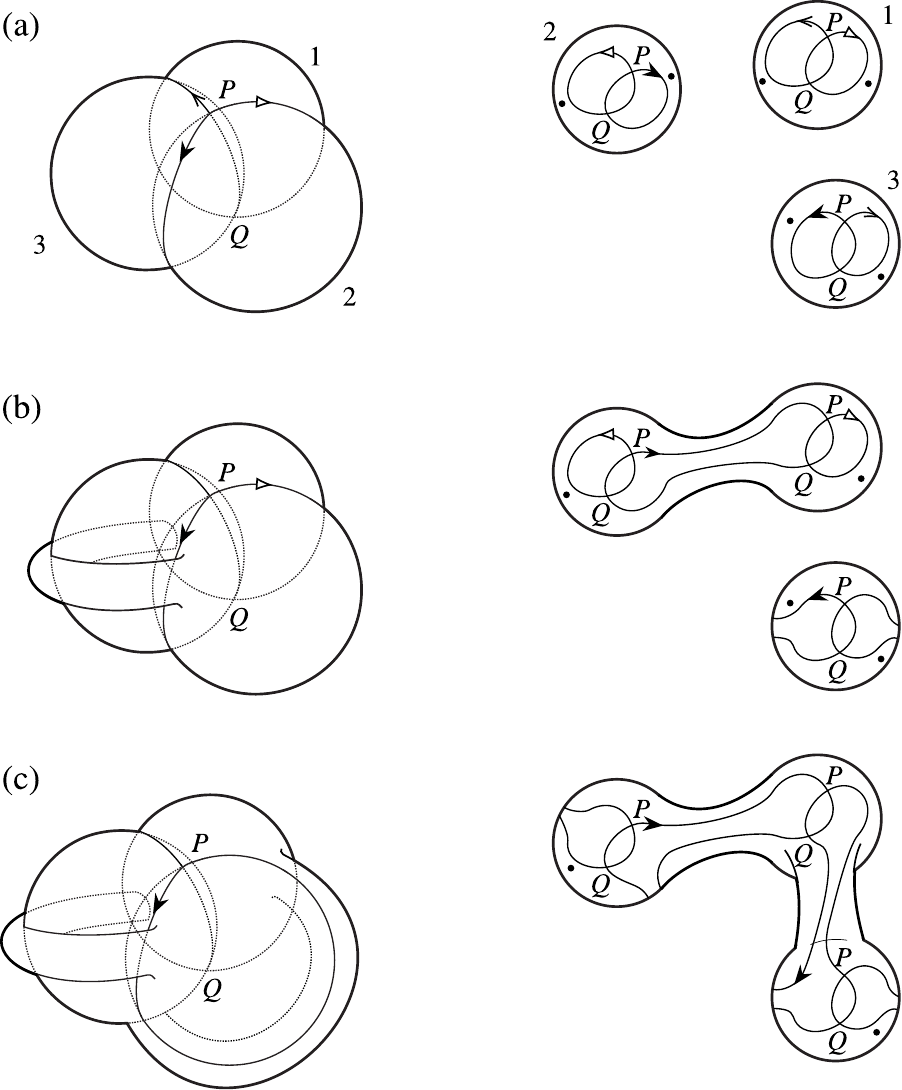}
  \caption{Constructing Banchoff's sphere: (a) starting from a bunch of 
    three $2$-spheres, (b, c) we add tubes to build a Dehn sphere. In the right 
    column we can see how the corresponding diagram changes.}
  \label{fig:figs_banchoffTrefoil}
\end{figure}

\begin{figure}
  \noindent%
  \includegraphics{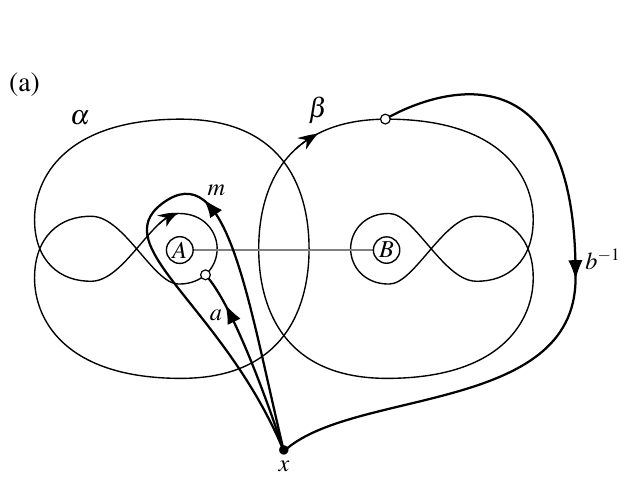}%
  \hfill%
  \includegraphics{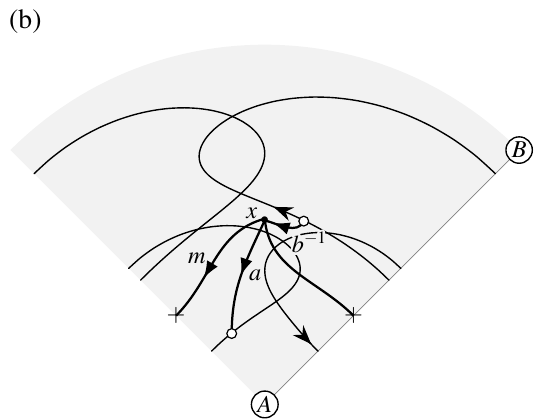}%
  
  \caption{(a) Banchoff's sphere diagram with the loops selected to generate 
    the fundamental group of the trefoil knot complement.
    (b) The fan $\Delta$ obtained cutting previous diagram along the thin 
    horizontal line between point $A$ and $B$.}
  \label{fig:banchoff}
\end{figure}

Let $K$ be the trefoil knot lying in $S^3$ with a $2\pi/3$
rotational symmetry as in Figure~\ref{fig:trefoil-banchoff}(a).
A filling Dehn sphere in $S^3$ that
diametrically splits $K$ can be constructed as follows. Look at the embedded 
$2$-sphere of Figure~\ref{fig:trefoil-banchoff}(a),
whose interior intersects $K$ in an unknotted arc, and take another two 
copies of it, each one located at each ``petal'' of the trefoil knot. These 
three
embedded $2$-spheres in $S^3$ intersect themselves and $K$ as in
Figure~\ref{fig:trefoil-banchoff}(b), and they form a filling collection of 
spheres $\Sigma$ 
in $S^3$. Two Banchoff \emph{type 1 surgeries}~\cite{Banchoff} 
between the $2$-spheres transform this filling collection of spheres
into a filling Dehn sphere $\Sigma_B$ of $S^3$ (see~\cite{Shima}) which
is called \emph{Banchoff's sphere} in \cite{peazolibro,tesis}.
According to~\cite{peazolibro,tesis}, it is one of the three
unique filling Dehn spheres of $S^3$ with only two
triple points.
Moreover, if the two surgeries are taken following the knot $K$ as it is 
indicated in Figure~\ref{fig:trefoil-banchoff}(c), then $\Sigma_B$ 
diametrically splits $K$. Figure~\ref{fig:figs_banchoffTrefoil} shows how its 
Johansson diagram $\D$ is obtained from the singular set of $\Sigma$. The dots 
in the right-hand side of the picture represent the intersection of the surface 
with the trefoil knot. Figure~\ref{fig:banchoff}(a) shows the common 
representation of $\D$ where one point of the sphere has been sent to infinity.

Let $f:S^2\to S^3$ be a parametrization of $\Sigma_B$. 
The two curves $\alpha,\beta$ of the diagram $\D$ verify $\beta=\tau\alpha$, 
and they must be 
identified following the arrows in the obvious way. The preimages by $f$ of the 
two intersection points $A,B$ of $\Sigma_B$ with $K$ are the points also 
denoted 
by $A,B$ in $S^2$, see Figure~\ref{fig:banchoff}(a). When the notation does 
not lead to confusion, we will use 
the same names to the objects in $\Sigma_B$ and their preimages in $S^2$.

The fundamental group $\pi_1(\Sigma_B- \{A,B\},x)\simeq\pi_K$ based at 
the point $x$ is generated by the loops $m$ and $c$, where:
\begin{itemize}
  
  \item $m$ is the generator of $\pi_1(S^2-\{A,B\})$ depicted in 
    Figure~\ref{fig:banchoff}(a); and
  
  \item the loop $c=a\,b^{-1}=\alpha^{\#}$ dual to the curve $\alpha$ of $\D$,
    where $a$ and $b$ are the paths depicted in
    Figure~\ref{fig:banchoff}(a) joining $x$ with related points on $\alpha$ 
    and $\beta$ respectively.
    
\end{itemize}
Note that the loop $m$ in $\Sigma_B$ is homotopic to a meridian of 
$K$.

After computing for $\Sigma_B-\{A,B\}$ the presentation of its 
fundamental group given in \cite{spectrum}, we obtain
\begin{equation}
  \label{ec:pi_1(M-Trefoil)}
  \pi_K\simeq \langle\, m,c\mid mcm=cm^{-1}c \,\rangle.
\end{equation}
It is straightforward to see that this group is isomorphic to the standard 
presentations of the trefoil knot group.

\section{Cyclic branched covers over the trefoil knot}
\label{sec:branched-covers}

\subsection{Johansson diagrams and fundamental group}

Let $p:\widehat{M}_n\to S^3$ be the $n$-fold ($n<\infty$) cyclic covering of 
$S^3$ branched over $K$.
By Theorem~\ref{thm:filling-lift-to-filling}, Banchoff's sphere 
$\Sigma_B\subset S^3$ lifts to a filling Dehn sphere 
$\widehat{\Sigma}_B$ of $\widehat{M}_n$.

We use the same notation as in Section~\ref{sec:splitting}.
Take $p^{-1}(x)=\{x_1,\ldots,x_n\}$, 
and denote also by $\{x_1,\ldots,x_n\}$ the corresponding
points in $\widehat{S}$. Let $\rho:\pi_K\to \Omega_n$ be
the monodromy homomorphism associated with $p$.
Since $p$ is cyclic, $\rho$ sends $\pi_K$ onto a cyclic subgroup 
$C_n$ of $\Omega_n$. The fact that $C_n$ is abelian and the relation in 
\eqref{ec:pi_1(M-Trefoil)} imply that $\rho(c)=\rho(m)^3$.
Therefore $C_n=\langle\rho(m)\rangle$. Since $C_n$ must act transitively
on $\{x_1,\ldots, x_n\}$,
$\rho(m)$ must be a cycle of order $n$. If we
identify $\Omega_n$ with the permutation group of the subscripts 
$\{1,\ldots,n\}$ in the
natural way, renaming $\{x_1,\ldots,x_n\}$ if necessary,
we can assume that $\rho(m)=(1,2,\ldots, n)$. In the following paragraphs all 
the subscripts are considered modulo $n$.

The loop $m$ generates the fundamental group of $S-\{A,B\}$, 
which is infinite cyclic. Therefore, the monodromy homomorphism $\rho_S$ is
given by $m\mapsto (1,2,\ldots, n)$. If $m_i$ is the lift of $m$ to 
$\widehat{S}$ based at $x_i$, by the election of $\rho(m)$, the lifted path 
$m_i$ starting at $x_i$ must have its endpoint at $m_{i+1}$.
\begin{figure}
  \centering
  \includegraphics{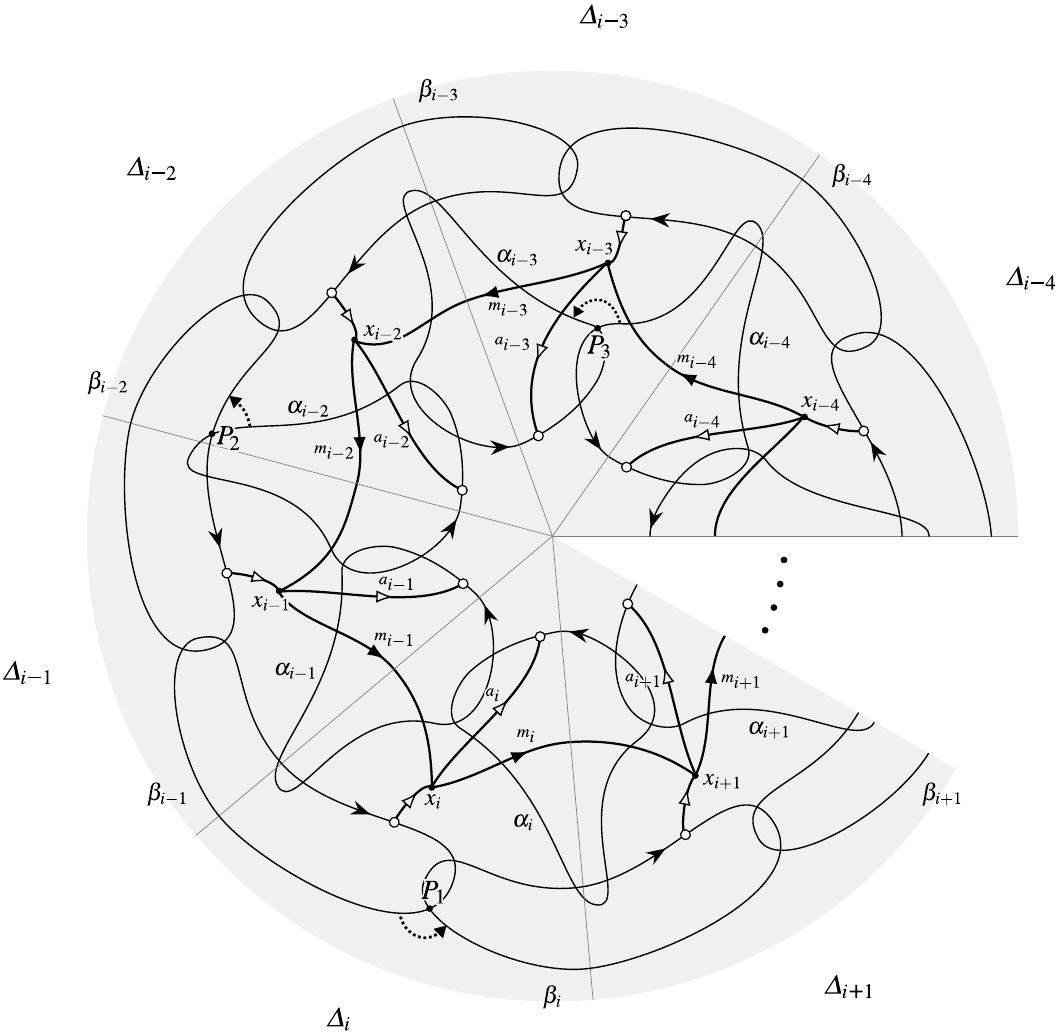}
  
  \caption{Building the diagram of a cyclic branched covering. The lifts of
    $m$ are the thick paths marked with triangle arrows. Those of 
    $a$ are the thick paths marked with triangle empty arrows. 
    The lifts of $b$ are unlabelled, but the path 
    $b_i^{-1}$ is the one ending at $x_i$ also marked 
    with an empty triangle arrow.}
  \label{fig:general_cyclic}
\end{figure}

After: (i) cutting the diagram of Figure~\ref{fig:banchoff}(a) along the line 
that connects $A$ and $B$ in the same figure; and (ii) sending $B$ to infinity; 
the \emph{fan} $\Delta$ of Figure~\ref{fig:banchoff}(b) is obtained. The domain 
$\widehat{S}$ of $\widehat{\Sigma}_B$ is obtained by cyclically gluing $n$ 
copies of $\Delta$. If $\Delta_i$ is the copy of $\Delta$ that contains $x_i$, 
$i=1,\ldots,n$, the diagram $\widehat{\D}$ of $\widehat{\Sigma}_B$ is built up 
with 
$\Delta_1,\ldots,\Delta_n$ glued together counterclockwise (the direction does 
not matter, but we choose it according to the direction of $m$ in the diagram 
to visualize it better, see Figure~\ref{fig:general_cyclic}).

As it is explained in Section~\ref{sec:splitting}, the lifts of $c$ to 
$\widehat{\Sigma}_B$ describe how the curves of the diagram $\widehat{\D}$ 
become 
identified in $\widehat{\Sigma}_B$. Let $a_i$, $b_i$ and $c_i$ be the lifts of 
$a$, 
$b$ and $c$ respectively based at $x_i$. Let $\alpha_i$ be the lift of $\alpha$ 
to $\widehat{S}$ at which $a_i$ has its endpoint. In the same way, let $\beta_i$ 
be 
the lift of $\beta$ to $\widehat{S}$ at which $b_i$ has its endpoint. With this 
notation, the situation in $\widehat{S}$ is as depicted in 
Figure~\ref{fig:general_cyclic}. The path $c_i$ connects $x_i$ with 
$x_{\rho(c)(i)}=x_{i+3}$, crossing a double curve of $\widehat{\Sigma}_B$. 
Hence, 
the curve at which $a_i$ ends must be identified with the one at which 
$b_{i+3}^{-1}$ starts and therefore $\tau\alpha_i = \beta_{i+2}$ (see 
Figure~\ref{fig:general_cyclic}). This allows 
us to proof:

\begin{lemma}[See~{\cite[Theorem~2.1]{CHK}}]
  The fundamental group of $\widehat{M}_n$ is isomorphic to the Sieradski group
  \[
    \mathcal{S}(n) = \langle\, g_1,\ldots,g_n \mid 
      \text{$g_i = g_{i-1}\,g_{i+1}$ for $i=1,\ldots,n$}\,\rangle,
  \]
  where the indices are taken modulo $n$.
\end{lemma}

\begin{proof}
  The fundamental group of $\widehat{M}_n$ coincides with that of 
  $\widehat{\Sigma}_B$. Since $\widehat{S}$ is a $2$-sphere, we can use the 
  presentation given in~\cite[Chp. 4]{peazolibro}. In this presentation, the 
  fundamental group is generated by the dual
  loops to $\alpha_i$ and $\beta_i$, $i=1,\ldots,n$, and the relators are given 
  by the triplets of $\widehat{\D}$. By construction, dual 
  loops to sister curves are inverse to each other, and so 
$\pi_1(\widehat{M}_n)$
  is generated just by the dual loops $\beta_1^{\#},\ldots,\beta_n^{\#}$.
 
  Let us determine the triplets of $\widehat{\D}$. Take the point $P_1$
  in the $i$-th fan $\Delta_i$ of 
  Figure~\ref{fig:general_cyclic}. The curves $\beta_{i-1}$ and
  $\beta_i$ intersect at $P_1$. Since $P_1$ is the third crossing
  after the arrow in $\beta_i$, it must be identified with the third crossing
  after the arrow in $\tau\beta_{i} = \alpha_{i-2}$, which is the point $P_2$ 
  in $\Delta_{i-2}$. In the same way, since
  $P_2$ is the fifth crossing after the arrow in $\beta_{i-2}$, it must be
  identified with the fifth crossing after the arrow in 
  $\tau\beta_{i-2}=\alpha_{i-4}$, which is the point $P_3$ 
  in $\Delta_{i-3}$. Finally,
  since $P_3$ is the second crossing after the arrow in $\alpha_{i-3}$,
  it must be identified with the second crossing after the arrow in
  $\tau\alpha_{i-3}=\beta_{i-1}$, which is, as expected, $P_1$.
  
  For each $j=1,2,3$ take a small path $\delta_j$ near $P_j$ as
  the dotted arcs in 
  Figure~\ref{fig:general_cyclic}, in such a way that the endpoint
  of $\delta_j$ is related by $\widehat{f}$ with the starting point of 
$\delta_{j+1}$, 
  where the subscripts are taken modulo $3$. Then the loop
    $(\widehat{f}\circ \delta_1)\,
      (\widehat{f}\circ \delta_2)\,(\widehat{f}\circ \delta_3)$
  is contractible in $\widehat{\Sigma}_B$. Hence, the product of 
  $\widehat{\D}$-dual loops
    $\beta_{i}^{\#}\,\beta_{i-2}^{\#}\,\alpha_{i-3}^{\#} =
      \beta_{i}^{\#}\,\beta_{i-2}^{\#}\,\left(\beta_{i-1}^{\#}\right)^{-1}$
  is also contractible (full details in~\cite[Chp. 4]{peazolibro}).
  Therefore $\beta_{i-1}^{\#}=\beta_{i}^{\#}\,\beta_{i-2}^{\#}$
  in $\pi_1(\widehat{M}_n,x)$.
  
  It is straightforward to see that all the relations are of this form. 
  In~\cite[Chp. 4]{peazolibro} it is proved that these are \emph{all} the 
  nontrivial relations. Taking $g_i=\beta_{n-i}^{\#}$, $i=1,\ldots,n$, the 
  presentation of $\mathcal{S}(n)$ of the statement is obtained.
\end{proof}

\begin{example}
  \begin{figure}
    \centering
    \includegraphics{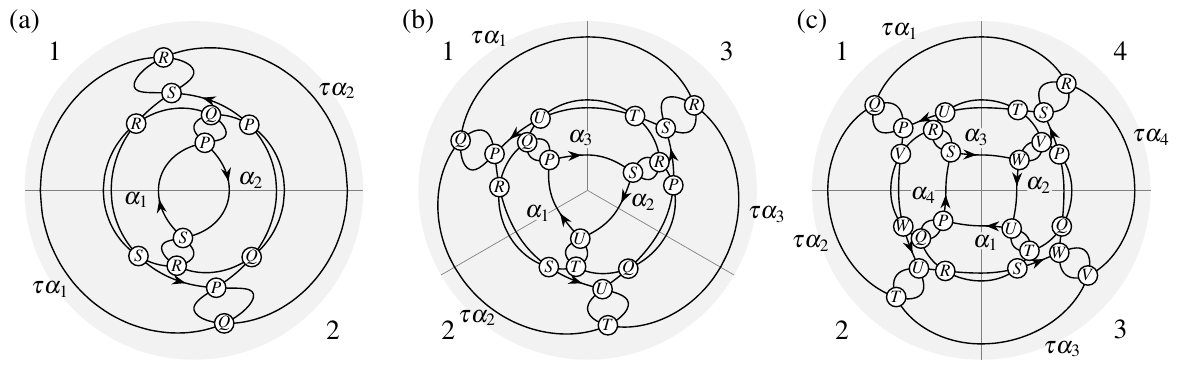}
    
    \caption{Diagrams for the first cyclic coverings branched over the trefoil:
      (a) The lens space $L(3,1)$ as the cyclic $2$-fold covering of $S^3$ 
      branched over $K$; (b) the $3$-fold covering; and (c) the $4$-fold 
      covering.
    }
    \label{fig:cyclic}
  \end{figure}
  Consider $p:\widehat{M}_2\to S^3$ given by the presentation $m\mapsto(1,2)$, 
  $c\mapsto(1,2)$, see Figure~\ref{fig:cyclic}(a). The previous descriptions 
  allows us to conclude that $\pi_1(M_2)=\Z_3$, in fact $\widehat{M}_2$ is 
$L(3,1)$. 
  Figure~\ref{fig:cyclic}(b) shows the diagram constructed for $\widehat{M}_3$ 
given by 
  the presentation $m\mapsto(1,2,3)$, $c\mapsto 1_{\Omega_3}$. The 
  fundamental group is isomorphic to the group of the quaternions $Q_8$. Hence, 
  $\widehat{M}_3$ is a prism manifold, which are characterized by their 
  fundamental 
  group. According to the notation of~\cite{HKMR} this is $M(2,1)$, also called 
  the Quaternionic Space~\cite{montetese}. Consider now the presentation 
  $m\mapsto(1,2,3,4)$, $c\mapsto (1,4,3,2)$, the fundamental group of 
$\widehat{M}_4$ is 
  $SL_2(\Z_3) \cong \Z_3\rtimes Q_8$. This is a tetrahedral manifold, which are 
  also characterized by their fundamental group. In this case the manifold is 
  the Octahedral space~\cite{montetese}.
\end{example}

\subsection{The Sieradski complex}

\begin{figure}
\centering
  \includegraphics{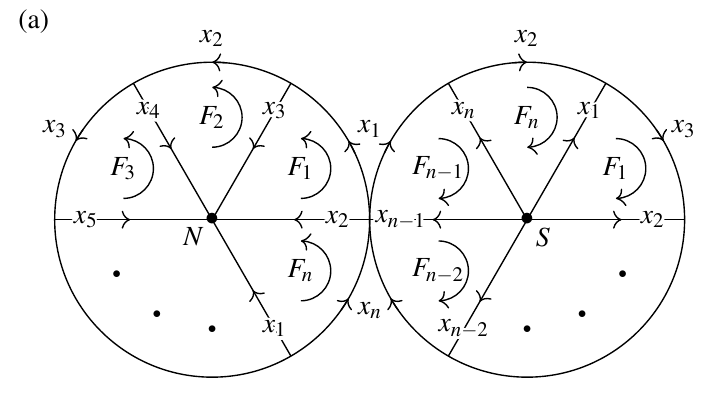}
  \hfill
  \includegraphics{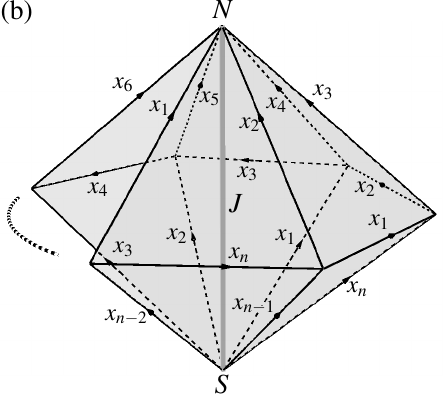}

  \caption{The Sieradski complex on the sphere $S^2=\partial D^3$.}
  \label{fig:sieradskiCn}
\end{figure}
The family of polyhedra with identified faces depicted in 
Figure~\ref{fig:sieradskiCn} was introduced in \cite{sieradski}. The quotient 
spaces of these polyhedra is a family of $3$-manifolds $M_n$, 
with $n\geq 2$, whose fundamental 
groups are the Sieradski groups $\mathcal{S}(n)$. In~\cite{CHK} it is proved 
that $M_n$ is the $n$-fold cyclic branched cover over the trefoil knot 
$\widehat{M}_n$. As it has been shown above, Sieradski groups naturally appear 
in our 
construction, so it is natural to expect that 
the lifts of Banchoff sphere have some relation with Sieradski 
polyhedra (see Figure~\ref{fig:sieradskiCn}). In fact, Banchoff sphere allows us 
to give an alternative 
proof of Theorem~2.1 of~\cite{CHK}.
\begin{figure}
  \centering
  \subfigure[]{
    \includegraphics[width=2cm]{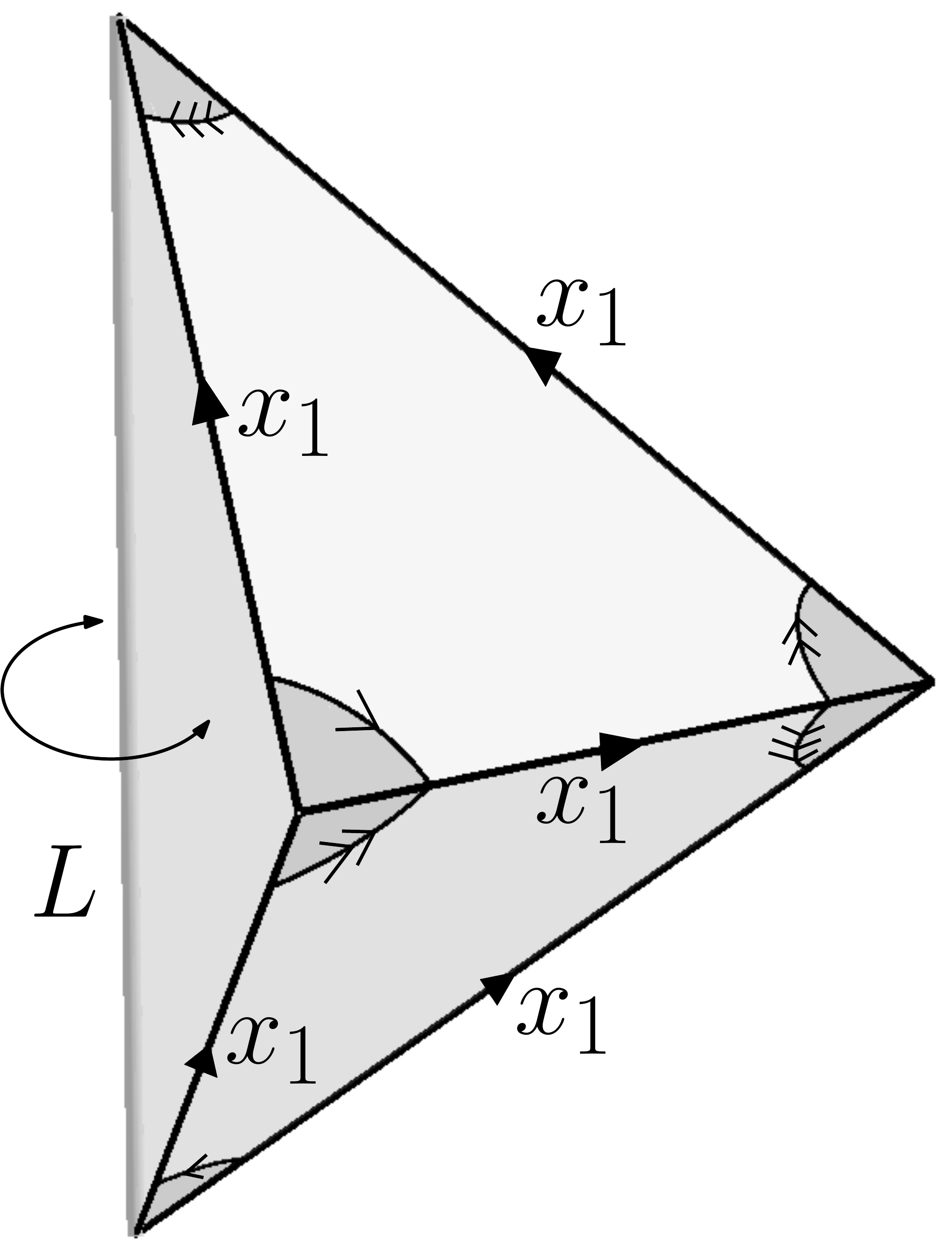}
  }
  \hfill
  \subfigure[]{
    \includegraphics[width=2cm]{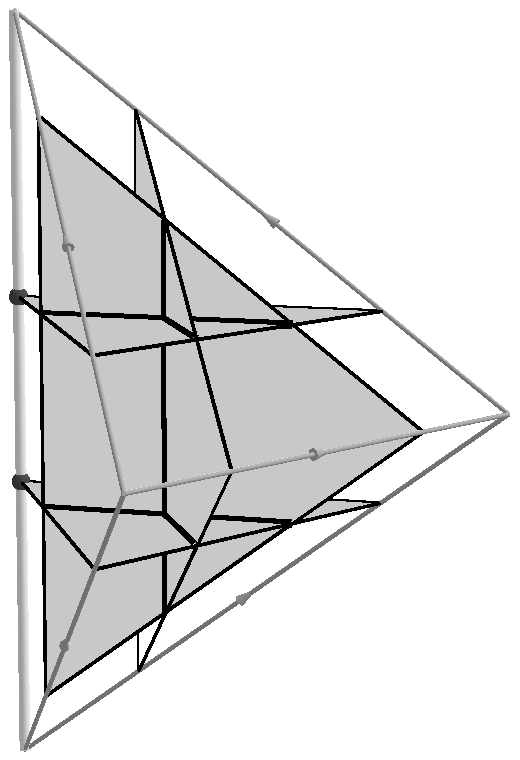}
  }
  \hfill
  \subfigure[]{
    \includegraphics[width=65mm]{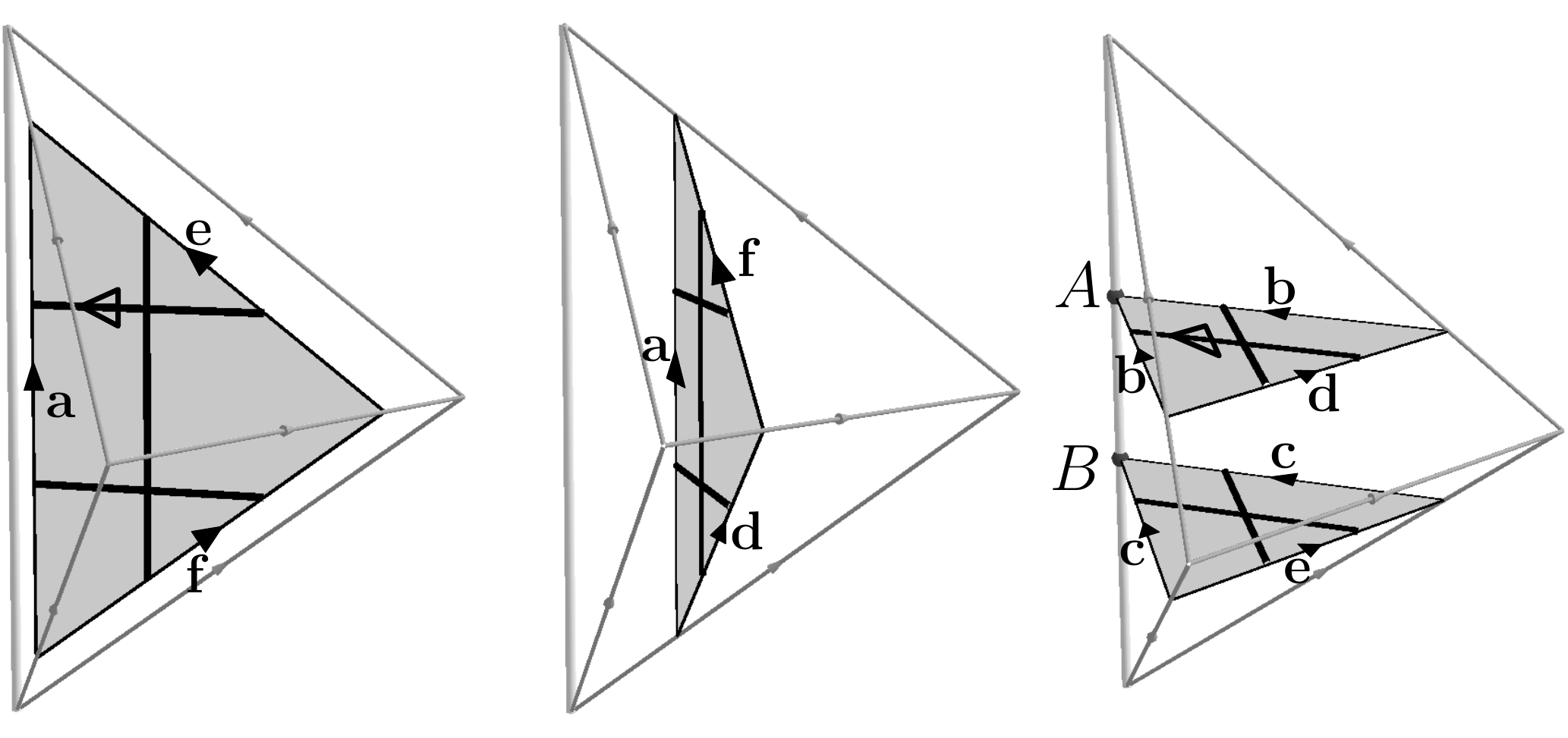}
  }
  
  \subfigure[]{
    \includegraphics[height=6cm]{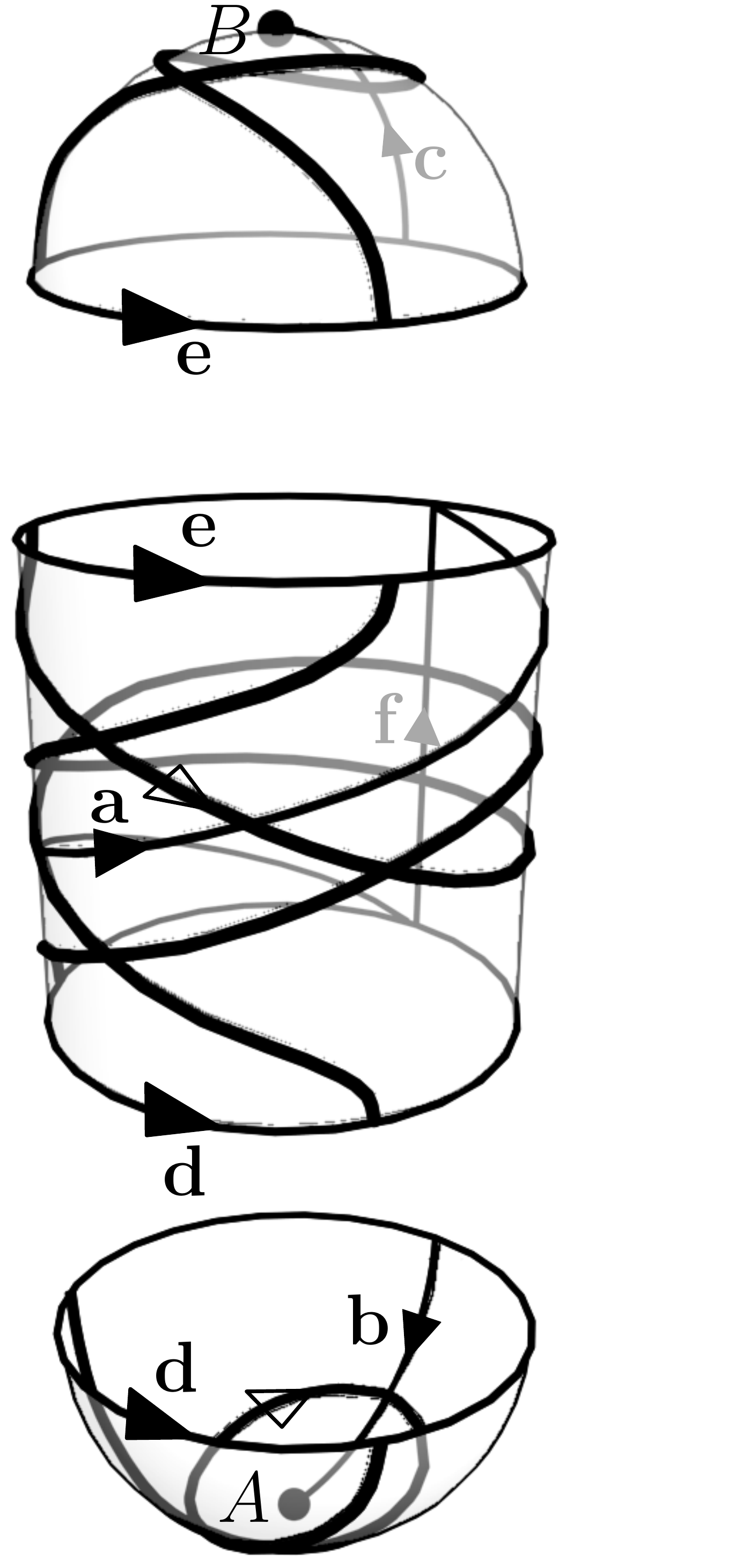}
  }
  \hspace{3cm}
  \subfigure[]{
    \includegraphics[height=35mm]{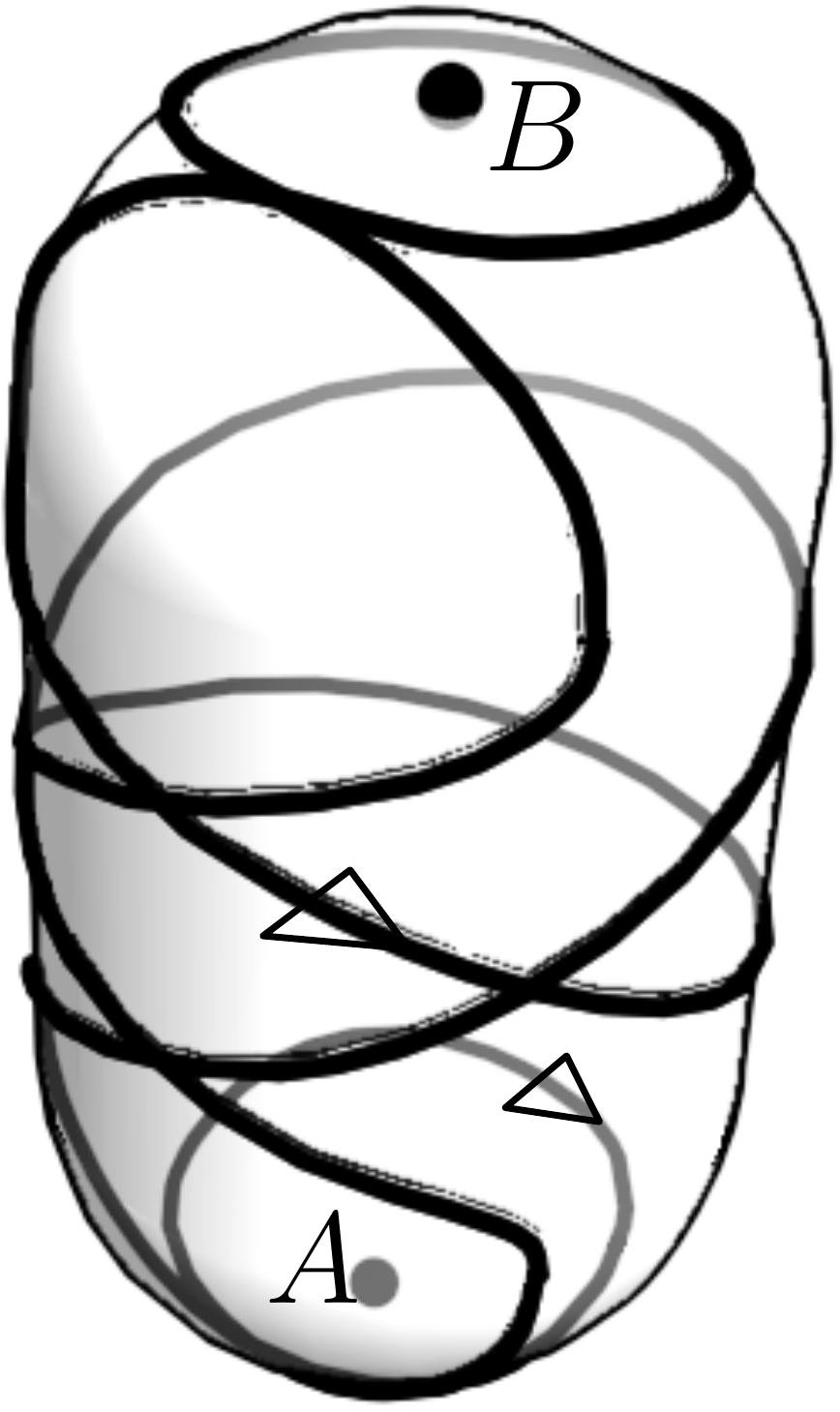}
  }

  \caption{Banchoff's sphere splitting the trefoil in Sieradski complex for 
    $n=1$.}
  \label{fig:Sieradski-bipyramid}
\end{figure}

\begin{theorem}[\cite{CHK}]
 For each $n=2,3,\ldots$ the $3$-manifold $M_n$ is the $n$-fold cyclic cover of 
 $S^3$ branched over the trefoil.
\end{theorem}

\begin{proof}
  Fix an $n=2,3,\ldots$. For the Sieradski polyhedron $\mathcal{P}_n$,
  consider the $2\pi/n$ rotation $r$ around the vertical central axis $J$ 
  connecting the ``north and south poles'' $N$ and $S$ of $\mathcal{P}_n$.
  By the symmetry of the identification of points on 
  $\partial \mathcal{P}_n$, $r$ preserves these identifications ($r$ sends 
  identified points to identified points on $\partial \mathcal{P}_n$), and 
  therefore $r$ defines a homeomorphism of $M_n$. Moreover,
  the group generated by 
  $r$ is cyclic of order $n$. The quotient space of the pair $(M_n,J)$ under 
  the action of $\langle r \rangle$ is the pair $(M_1,L)$, where $M_1$ is the 
  manifold obtained after identifying the faces of the tetrahedron
  $\mathcal{P}_1$ of Figure~\ref{fig:Sieradski-bipyramid}(a) 
  in the following way: the two vertical faces must be identified by a rotation 
  around the vertical edge; and the other two faces must be identified in the 
  unique way in which the boundary edges become identified as indicated by the 
  arrows. Let $L$ denote the image of the vertical edge of the tetrahedron 
  after the identification.
  
  \begin{claim}
    The pair $(M_1,L)$ is homeomorphic to $(S^3,K)$.
  \end{claim}

  \begin{proofClaim}
    Take the immersed surface $\Sigma$ in $M_1$ which is the projection of the 
    surface depicted inside the tetrahedron in 
    Figure~\ref{fig:Sieradski-bipyramid}(b, c). The four pieces of this 
    surface become glued by the identification as it is indicated in 
    Figure~\ref{fig:Sieradski-bipyramid}(d, e), and the intersection of $L$ 
    with $\Sigma$ corresponds to the points $A$ and $B$ of the same figure.
    After some ambient isotopies, the diagram of $(\Sigma,L)$ 
    becomes the one of Figure~\ref{fig:banchoff}(a). It is not 
    difficult to check 
    that $\Sigma$ fills $M_1$. Therefore, $M_1$ is $S^3$ and $\Sigma$ is 
    Banchoff's sphere. Since the intersection of $L$ and $\Sigma$ coincides 
    with the intersection of the trefoil and $\Sigma_B$, by
    Proposition~\ref{prop:diagram-represent-knots}
    we conclude also that $L$ is the trefoil knot.\qed
  \end{proofClaim}
  
  The unique points of $M_n$ that become fixed by $r$ are those on $J$. 
  Therefore, the covering of $M_n$ over $M_1$ defined by $\langle r\rangle$ is 
  a $n$-fold covering of $S^3$ branched over $K$. 
  Since the group $\langle r\rangle$ of deck transformations is cyclic of 
  order $n$, it turns out that $M_n$ is the $n$-fold cyclic covering of $S^3$ 
  branched over the trefoil knot.
\end{proof}

The lift of Banchoff's sphere to $M_n$ can be built inside 
$\mathcal{P}_n$ by cyclically gluing $n$ copies
of the pieces of the surface of Figure~\ref{fig:Sieradski-bipyramid}(b) around
the axis $J$ (Figure~\ref{fig:sieradskiCn}(b)).

\section{Other examples}\label{sec:other-examples}

\subsection{Locally cyclic branched covers}
\label{sub:trefoil-locally-cyclic}

Since all the locally cyclic coverings of 2 and 3 sheets of $S^3$ branched over 
$K$ are in fact cyclic, the first non-cyclic example $p:M\to S^3$ is the one 
given by the representation of $\pi_k$ into $\Omega_4$ defined by
$m\mapsto(1,2,3,4)$ and $ c\mapsto(1,2)$.
The construction of the diagram is as in the cyclic case.
Since the image of $m$ is a cycle of maximal length the 
lift of the Dehn surface has $S^2$ as domain, and $\rho(c)$ describes 
how to identify the curves of the diagram. The final diagram is the one of 
Figure~\ref{fig:locallycyclic}, which gives
\[
  \pi_1(M) \cong \langle\, \alpha_1,\alpha_2,\alpha_3,\alpha_4 \mid
    \alpha_1\alpha_2^{-2},
      \alpha_1\alpha_3\alpha_4^{-1},
        \alpha_1\alpha_3^{-1}\alpha_4,
          \alpha_4\alpha_2^{-1}\alpha_3\,
  \rangle\cong\Z_3\rtimes\Z_8.
\]
Therefore $M$ is the prism manifold $M(3,2)$ in the notation of~\cite{HKMR}.
Compare it with the $4$-fold cyclic covering depicted in 
Figure~\ref{fig:cyclic}(c).
\begin{figure}
  \centering
  \includegraphics{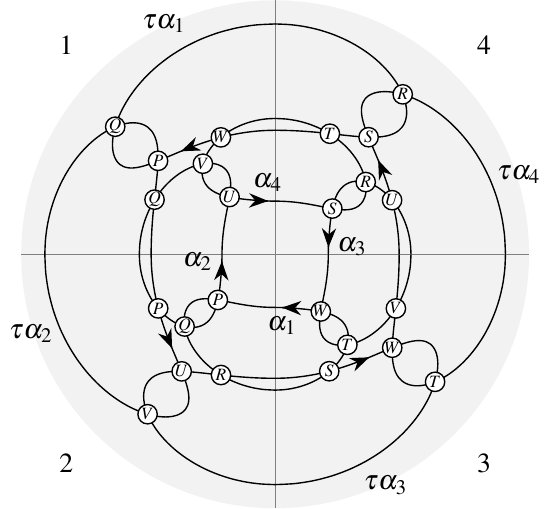}
  
  \caption{The diagram for the locally cyclic covering given by the
    presentation $m\mapsto (1,2,3,4)$, $a\mapsto (1,2)$.}
  \label{fig:locallycyclic}
\end{figure}

\subsection{The $3$-fold irregular cover}
\label{sub:trefoil-non-cyclic}

Let $p:\widehat{M}\to S^3$ be an $n$-fold branched covering over the trefoil 
$K$ with $n<\infty$.
In the previous sections we have seen how to construct the Johansson
diagram of a filling Dehn sphere of $\widehat{M}$ when
$\rho(m)$ is a cycle of order $n$, where $\rho$ is the monodromy of $p$.
It is also possible to construct the Johansson diagram of a filling Dehn sphere 
of $\widehat{M}$ in the general case, even if $\rho(m)$ is not a cycle of 
maximal length. We illustrate this case with the simplest example: the $3$-fold 
irregular covering over the trefoil. We use the same notation as in previous 
sections.

\begin{figure}
  \centering
  \includegraphics{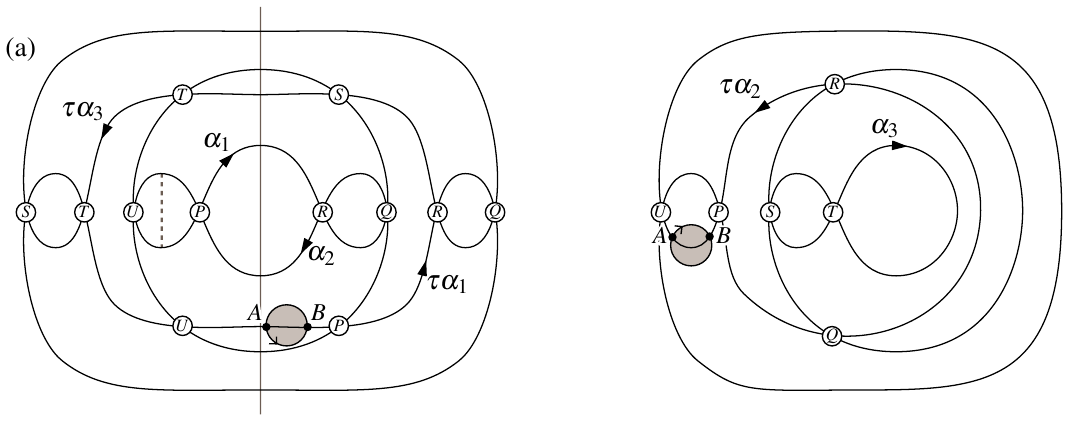}
  
  \includegraphics{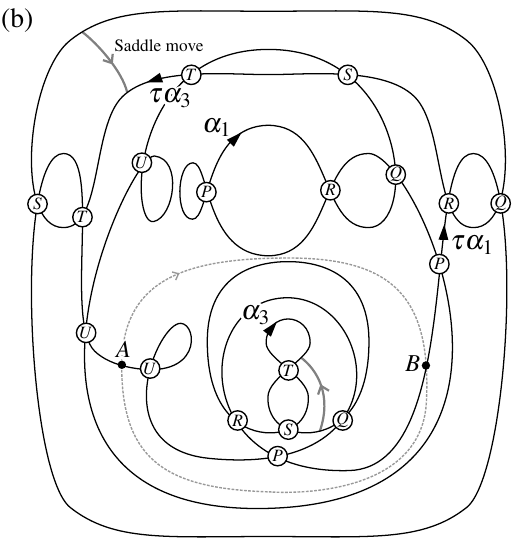}
  \hfill
  \includegraphics{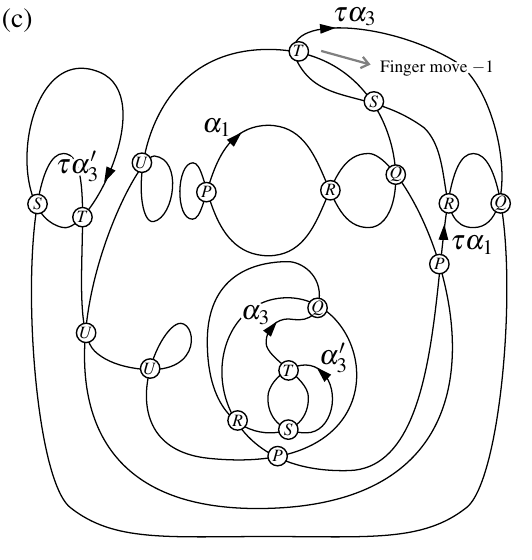}
  
  \includegraphics{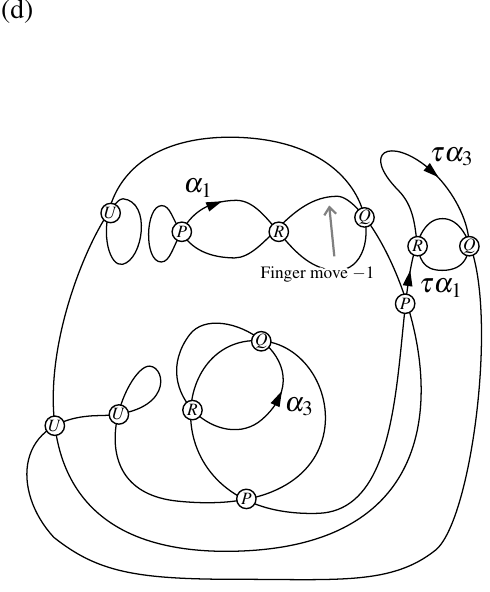}
  \hfill
  \includegraphics{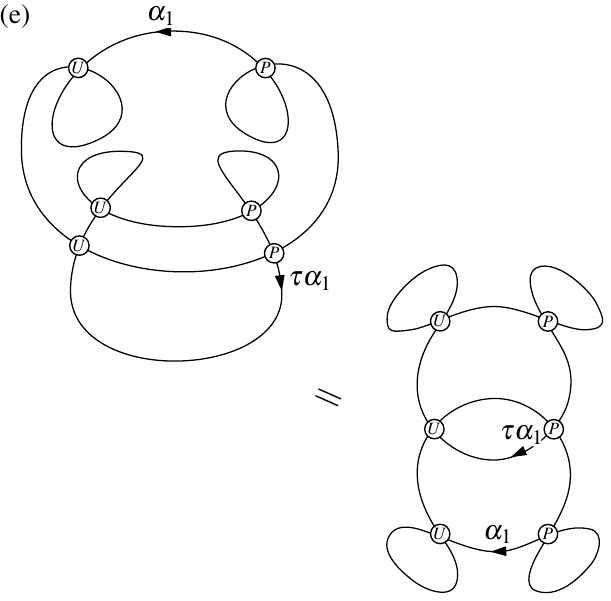}
  
  \caption{The first non-locally cyclic cover of $S^3$ branched over $K$ is 
    $S^3$.}
  \label{fig:irregular}
\end{figure}
Assume that $n=3$, and that $\rho(m)=(12)$.
The lifts $m_1$, $m_2$ and $m_3$ of $m$ to the domain surface $\widehat{S}$ of 
$\widehat{\Sigma}=p^{-1}(\Sigma_B)$ verify
\begin{itemize}
 \item $m_1$ connects $x_1$ with $x_2$;
 \item $m_2$ connects $x_2$ with $x_1$; and
 \item $m_3$ connects $x_3$ with itself. 
\end{itemize}
Therefore, $\widehat{S}$ is a disjoint union of two $2$-spheres. One of them, 
$S_{12}$, contains $x_1$ and $x_2$ and can be obtained by gluing two copies
$\Delta_1$ and $\Delta_2$ of the fan $\Delta$. The other 
one, containing $x_3$ must be a copy of the domain surface of $\Sigma_B$. The 
restriction of
$p_S$ to $S_{12}$ is a $2$-fold branched covering with branching set $\{A,B\}$,
and $p_S|_{S_3}$ is a $1$-fold branched covering with branching set $\{A,B\}$, 
hence a homeomorphism. The diagram $\widehat{\D}$ of $\widehat{\Sigma}$
is the lift to $\widehat{S}=S_{12}\sqcup S_3$ of the diagram $\D$
through $p_S$. By the same arguments of previous sections, $\widehat{\D}$ in
$S_{12}$ looks like the left-hand side of Figure~\ref{fig:irregular}(a).
The diagram $\widehat{\D}$ in $S_{3}$ looks exactly like the diagram
$\D$ in $S$, except for the identification of the curves, that will be
determined by the element $\rho(c)$ given by the monodromy homomorphism.
Set $\Sigma_{12}=\widehat{f}(S_{12})$ and $\Sigma_{3}=\widehat{f}(S_{3})$.

Since $\rho(m)$ and $\rho(c)$ verify the identity
\[
  \rho(m)\,\rho(c)\,\rho(m)=\rho(c)\,\rho(m)^{-1}\,\rho(c)
\]
and the subgroup of $\Omega_3$ generated by $\rho(m)$ and $\rho(c)$
acts transitively on the set $\{1,2,3\}$, 
then $\rho(c)=(1,3)$ or $\rho(c)=(2,3)$.
Assume that $\rho(c)=(2,3)$ (the case $\rho(c)=(1,3)$ is equivalent).
The endpoint of $a_1$ is related by $\widehat{f}$ with
the endpoint of $b_1$, the endpoint of $a_2$ is related by $\widehat{f}$ with
with the endpoint of $b_3$, and the endpoint of $a_2$ 
is related by $\widehat{f}$ with the endpoint of $b_3$. The resulting
sistering of $\widehat{\D}$ is indicated in 
Figure~\ref{fig:irregular}(a) (in this 
figure we depict $S_3$ as the fan $\Delta_3$ glued with itself, compare it 
with Figure~\ref{fig:banchoff}).
\begin{figure}
\centering
\subfigure[]{\includegraphics[width=0.45\textwidth]{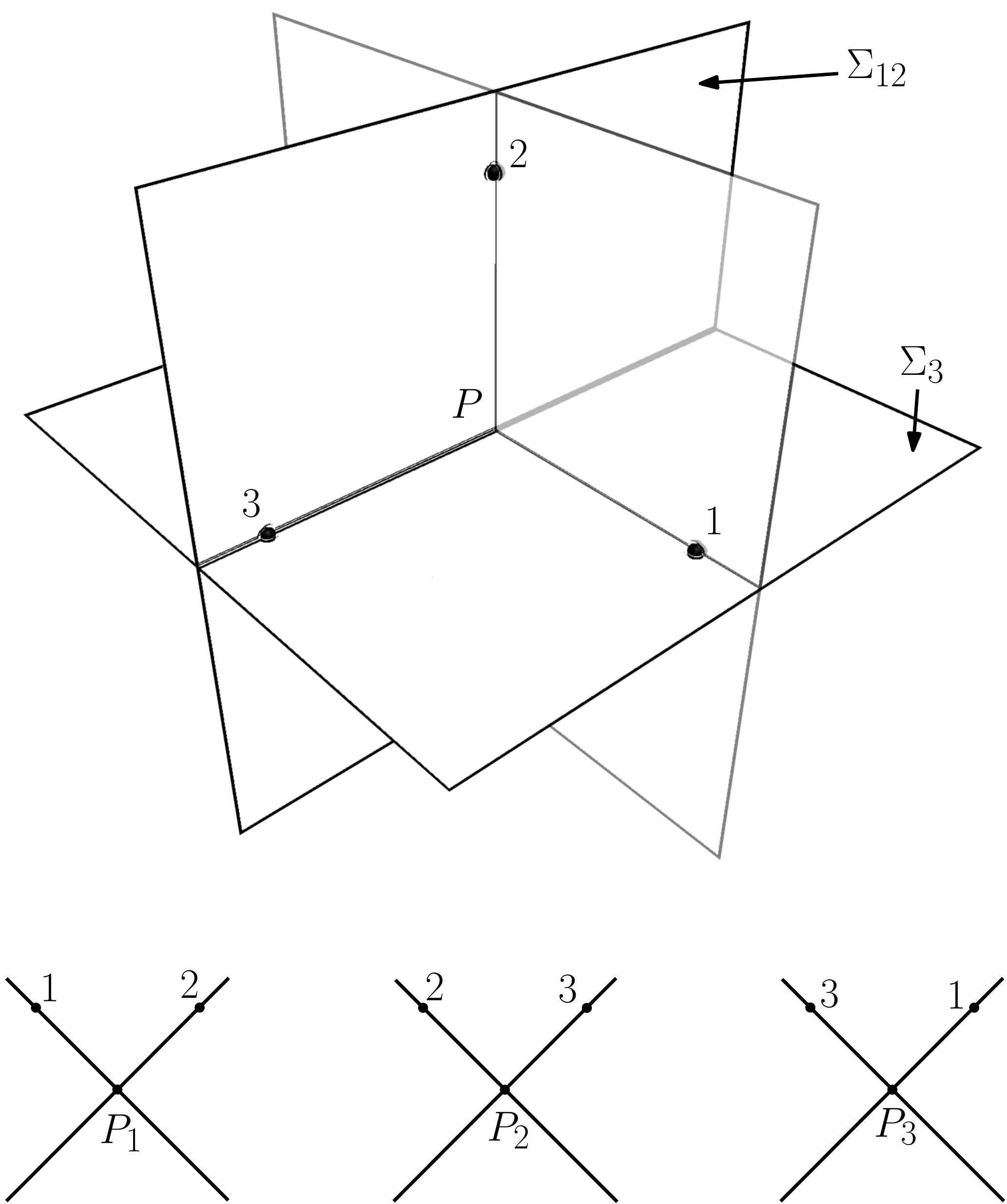}
\label{fig:irregular-piping-01} }
\hfill
\subfigure[]{\includegraphics[width=0.45\textwidth]{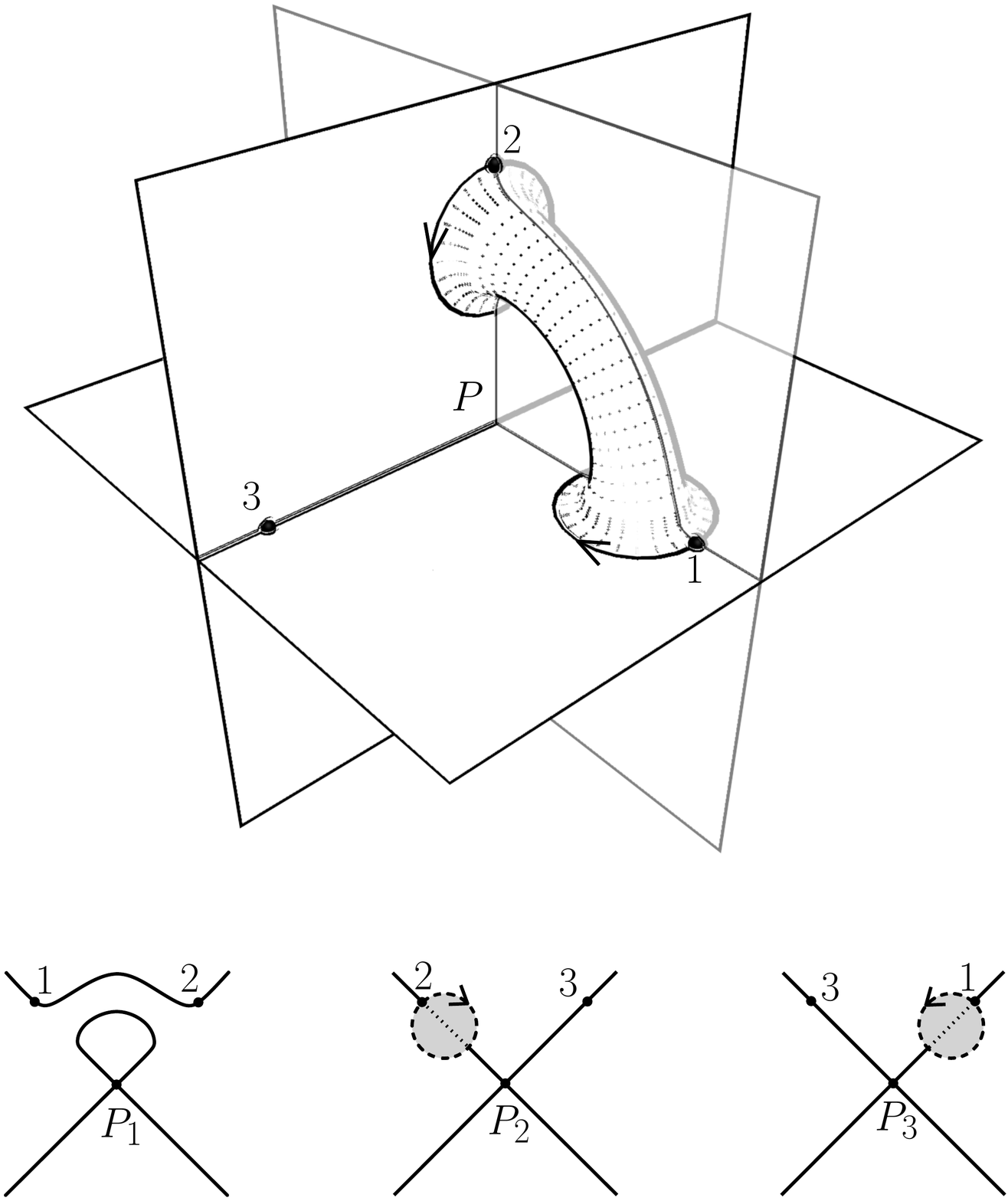}
\label{fig:irregular-piping-02} }

\caption{Surgery around a triple point.}
\label{fig:irregular-piping}
\end{figure}

Now, we modify $\widehat{\Sigma}$ by a surgery operation near a triple point 
where a sheet of $\Sigma_{12}$ and a sheet of $\Sigma_3$ 
intersect. We perform a ``piping'' between $\Sigma_{12}$ and $\Sigma_3$ around 
the triple point $P$ as indicated in the top part of Figure~\ref{fig:irregular-piping}.
The local effect of this operation on the diagram is depicted in the bottom 
part of the same picture. The resulting Dehn surface 
$\widetilde{\Sigma}$ has as domain the connected sum of $S_{12}$ and $S_3$, hence 
it is also a Dehn sphere. Its Johansson diagram appears in
Figure~\ref{fig:irregular}(b).

\begin{proposition}
  The Dehn sphere $\widetilde{\Sigma}$ fills $\widehat{M}$.  
\end{proposition}

\begin{proof}
  In order to check that $\widetilde{\Sigma}$ fills $\widehat{M}$
  it must be proved that all its edges, faces and regions 
  are open $1$-, $2$- and $3$-dimensional disks, respectively.
  The diagram of $\widetilde{\Sigma}$ implies that the edges and faces of
  $\widetilde{\Sigma}$ verify this requirement.
  
  The embedded $2$-sphere $\widehat{\Sigma}_3$ is 
  nulhomotopic because $\Sigma_B$ is nulhomotopic, and therefore 
  $\widehat{\Sigma}_3$ is separating in $\widehat{M}$. This implies that 
  the surgery that transforms $\widehat{\Sigma}$ into $\widetilde{\Sigma}$ connects
  two regions $R_1$ and $R_2$ of $\widehat{\Sigma}$ on one connected component of 
  $\widehat{M}-\widehat{\Sigma}_3$ with another two regions $R_3$ and $R_4$ of 
  $\widehat{\Sigma}$ on the other connected component of 
  $\widehat{M}-\widehat{\Sigma}_3$, creating two regions of $\widetilde{\Sigma}$.
  If $R_1=R_2$, there would be a loop $\lambda$ in $\widehat{M}$ that intersects
  $\widehat{\Sigma}$ transversely only at one non-singular point of 
  $\widehat{\Sigma}$, and in this case $p\circ \lambda$ would intersect
  $\Sigma_B$ transversely only at one non-singular point of 
  $\Sigma_B$; but this cannot happen because $\Sigma_B$
  is nulhomotopic (as any Dehn sphere in $S^3$). Hence,
  $R_1\neq R_2$, and the same argument gives $R_3 \neq R_4$.
  Therefore, the four regions of $\widehat{\Sigma}$ that become connected
  in pairs by the surgery are all different, and this implies that all the 
  regions of $\widetilde{\Sigma}$ are open $3$-balls.
\end{proof}

The Johansson diagrams of two filling Dehn spheres of the
same $3$-manifold are related by a sequence of 
\emph{\textbf{f}-moves}~\cite{peazolibro,RHomotopies,tesis},
provided that both filling Dehn spheres are nulhomotopic
(an equivalent set of moves is proposed in~\cite{Amendola09}).
It is natural to ask if \emph{\textbf{f}}-moves suffice to prove
the following well-known result about the $3$-manifold $\widehat{M}$ 
(see~\cite{peazolibro} for more details on \emph{\textbf{f}}-moves).

\begin{theorem}[\cite{Burde,fox,GH,MontesinosTesis}%
  \footnote{Of course, these and other similar questions became automatically 
  solved after Perelman's proof of the Poincar\'e Conjecture~\cite{perelman}.}]
  The $3$-fold irregular branched covering of the trefoil is the $3$-sphere.
\end{theorem}

\begin{proof}
  Starting from the diagram $\widetilde{\D}$ of Figure~\ref{fig:irregular}(b), 
  after a \emph{saddle move} the diagram $\D_1$ of 
  Figure~\ref{fig:irregular}(c) is obtained. After a \emph{finger move} $-1$ 
  (and ambient isotopies) we arrive to the diagram $\D_2$ of 
  Figure~\ref{fig:irregular}(d), and another finger move $-1$ finally gives the 
  diagram $\D_3$  of Figure~\ref{fig:irregular}(e), which coincides with the 
  diagram of \emph{Johansson's sphere}, another well known filling Dehn sphere 
  of $S^3$ (see~\cite{peazolibro,knots}). Since $\D_3$ is filling, $\D_2$ is 
  also filling because it is obtained applying a finger move~$+1$ to $\D_3$ 
  (see~\cite[Lem.~5.6 and Thm.~5.20]{peazolibro}). The same argument applies to 
  conclude that $\D_1$ is also filling. Since $\widetilde{\D}$ and $\D_1$ are 
  filling diagrams, the saddle move that relate them is an 
  \emph{\textbf{f}}-move. By~\cite[Thm.~5.20]{peazolibro}, $\widehat{M}$ is 
  $S^3$.
\end{proof}

\end{document}